\documentclass[11pt,a4paper]{article}
\usepackage{amsmath,amsfonts,amssymb,amsthm}
\title{H\"older continuity of normal cycles and of support measures of convex bodies}

\author{Daniel Hug and Rolf Schneider}
\date{}
\newcommand{\R}{{\mathbb R}}
\newcommand{\Sn}{{\mathbb S}^{n-1}}
\newcommand{\Kn}{{\mathcal K}^n}

\newcommand{\N}{{\mathbb N}}

\newcommand{\Ha}{{\mathcal H}}
\newcommand{\D}{{\rm d}}
\newcommand{\fed}{\,\rule{.1mm}{.20cm}\rule{.20cm}{.1mm}\,}

\newcommand{\RR}{\ensuremath{\mathbb{R}}}
\newcommand{\NN}{\ensuremath{\mathbb{N}}}

\newcommand{\rmwedge}{\mbox{$\bigwedge$}}

\DeclareMathOperator{\Nor}{Nor}

\newtheorem{theorem}{Theorem}
\newtheorem{lemma}{Lemma}[section]

\catcode`@=11
\def\section{
\setcounter{equation}{0} \setcounter{theorem}{0} \@startsection
{section}{1}{\z@}{-4.0ex plus -1ex minus
    -.2ex}{2.3ex plus .2ex}{\bf\Large}}
\def\subsection{\@startsection{subsection}{2}{\z@}{-3.25ex plus-1ex
    minus-.2ex}{1.5ex plus.2ex}{\reset@font\bf\Large}}

\begin{document}

\maketitle

\begin{abstract}
We provide an estimate of the distance (in the dual flat seminorm) of the normal cycles of convex bodies with given Hausdorff distance. 
We also give an estimate (in the bounded Lipschitz metric) of the support measures of convex bodies.
\\[1mm]
\end{abstract}


\section{Introduction}\label{sec1}

In 1986, Martina Z\"ahle \cite{Zae86} presented a current representation of Federer's curvature measures. For $k=0,\dots,n-1$, she defined a differential form $\varphi_k$ of degree $n-1$ on the Euclidean space $\R^{2n}$ such that, for each compact set $K$ of positive reach in $\R^n$ and for each Borel set $\beta$ in $\R^n$, the evaluation of ${\bf 1}_\beta \varphi_k$ at the normal cycle of $K$ yields the $k$th curvature measure of $K$, evaluated  at $\beta$. The curvature measures had previously been introduced by Federer \cite{Fed59} in a different way. The approach using currents has later been investigated, applied and considerably extended in work of Z\"ahle \cite{Zae87, Zae90}, Rataj and Z\"ahle \cite{RZ01, RZ03, RZ05}, Fu \cite{Fu94}, Pokorn\'{y} and Rataj \cite{PR12}, and others. The normal cycle has thus become an important tool for the treatment of curvature properties of very general classes of sets, and it has found many applications in integral geometry. In this note, we restrict ourselves to convex bodies. The normal cycle $T_K$ of a convex body $K$ in $\R^n$ (we recall the definition in Section \ref{sec2}) has a useful continuity property. If $K_i$, $i\in\NN$, and $K$ are convex bodies in $\R^n$ and $K_i\to K$ in the Hausdorff metric, as $i\to\infty$, then $T_{K_i}\to T_K$ in the dual flat seminorm for currents. This was stated without proof in \cite[p.~251]{Zae90} and was proved in \cite[Thm.~3.1]{RZ01}; see also \cite[Thm.~3.1]{Fu10}. The continuity property is used, for example, in \cite{HS13} in the course of the proof for a classification theorem for local tensor valuations on the space of convex bodies. 

The purpose of this note is to obtain a quantitative improvement of the preceding continuity result, in the form of a H\"older estimate. To formulate it, we denote by $\Kn$ the space of convex bodies (nonempty compact convex subsets) in Euclidean space  $\R^n$, as usual equipped with the Hausdorff metric $d_H$. We write $B^n$ and $\Sn$ for the unit ball and the unit sphere, respectively, in $\R^n$, and $\Ha^k$ for the $k$-dimensional Hausdorff measure. For $\rho>0$, $K_\rho:=K+\rho B^n$ is the parallel body of the convex body $K$ at distance $\rho$. Below, we identify $\R^n\times\R^n$ with $\R^{2n}$ and denote by $\mathcal{E}^{n-1}(\R^{2n})=\mathcal{E}(\R^{2n},\bigwedge^{n-1}\R^{2n})$ the vector space of all differential forms of degree $n-1$ on $\R^{2n}$ with real coefficients and of class $C^\infty$.

\begin{theorem}\label{T1}
Let $K,L\in\Kn$, and let $M\subset \R^{2n}$ be a compact convex set containing $K_1\times\Sn$ and $L_1\times\Sn$. Then, for each  $\varphi\in \mathcal{E}^{n-1}(\R^{2n})$, 
$$ |T_K(\varphi)-T_L(\varphi)|\le C(M,\varphi)\,d_H(K,L)^{\frac{1}{2n+1}},$$
w\textsc{}here $C(M,\varphi)$ is a constant which depends (for given dimension) on $M$ and on the Lipschitz constant and the sup-norm of $\varphi$ on $M$.
\end{theorem}

According to the definition of the dual flat seminorm, this result can be interpreted as local H\"older continuity of the normal cycles of convex bodies with respect to the Hausdorff metric and the dual flat seminorm. A similar, but essentially different quantitative result is obtained in \cite[Thm. 2]{CM06}. It refers to more general sets and is, therefore, less explicit. On the other hand, its restriction to convex bodies does not yield the present result, since at least one of the sets in \cite{CM06} has to be bounded by a submanifold of class $C^2$.

As mentioned, normal cycles are useful for introducing curvature measures for quite general classes of sets. In the theory of convex bodies, they have been used to introduce a generalization of curvature measures, the support measures. For $\R^n$, these are Borel measures on the product of $\R^n$ and the unit sphere $\Sn$, with the property that their marginal measures are the curvature measures on one hand and the surface area measures on the other hand. On the space of convex bodies with the Hausdorff metric, the support measures are weakly continuous. We improve this statement by showing that the support measures are locally H\"older continuous with respect to the bounded Lipschitz metric $d_{bL}$ (we recall its definition in Section 4). Let $\Lambda_i(K,\cdot)$ denote the $i$th support measure of $K\in\Kn$, normalized as explained in Section 4.

\begin{theorem}\label{T2}
Let $K,L\in\Kn$ be convex bodies, and let $R$ be the radius of a ball containing $K_2$ and $L_2$. Then
$$ d_{bL}(\Lambda_i(K,\cdot), \Lambda_i(L,\cdot)) \le C(R)\, d_H(K,L)^{1/2}$$
for $i\in\{0,\dots,n-1\}$, where $C(R)$ is a constant which (for given dimension) depends only on $R$.
\end{theorem}

An estimate of this type, though with a smaller exponent of the Hausdorff distance, could be derived directly from Theorem \ref{T1}. We shall obtain the stronger result of Theorem \ref{T2} by adapting an approach due to Chazal, Cohen--Steiner and M\'{e}rigot \cite{CCM10}.

A special case of Theorem \ref{T2} concerns the area measure $S_{n-1}(K,\cdot)$. If $\omega\subset\Sn$ is a Borel set, then $S_{n-1}(K,\omega)=2\Lambda_{n-1}(K,\R_n\times\omega)$. From Theorem \ref{T2} it follows under the same assumptions on $K$ and $L$ that
\begin{equation}\label{4.14} 
d_{bL}(S_{n-1}(K,\cdot),S_{n-1}(L,\cdot))\le C'(R)\,d_H(K,L)^{1/2}.
\end{equation}
We want to present some motivation for proving such an inequality.

The area measure is the subject of a famous existence and uniqueness theorem due to Minkowski (see, e.g., \cite[Sec.~7.1]{Sch93}). The uniqueness assertion has been improved by some stability results. One of these (going back to Diskant; see 
\cite[Thm.~7.2.2]{Sch93}) says that for convex bodies $K,L\in \Kn$ one has
\begin{equation}\label{4.15} 
d_H(K,L') \le \gamma \,\|S_{n-1}(K,\cdot)-S_{n-1}(L,\cdot)\|_{\rm TV}^{1/n}
\end{equation}
for a suitable translate $L'$ of $L$, where $\|\cdot\|_{\rm TV}$ denotes the total variation norm. Here $\gamma>0$ is a constant depending only on the dimension and on  a-priori bounds for the inradius and circumradius of $K$ and $L$. 

The stability result (\ref{4.15}) has the blemish that its assumption is too strong: the left side can be small even if the right side is large. For example, if $K$ is a unit cube and $L$ is a rotated image of $K$, arbitrarily close to $K$ but not a translate of it, then $\|S_{n-1}(K,\cdot)-S_{n-1}(L,\cdot)\|_{\rm TV}\ge 1$. It seems, therefore, more meaningful to replace the right-hand side in (\ref{4.15}) by an expression involving a metric for measures that metrizes the weak convergence. For the L\'{e}vy--Prokhorov metric, such a stability result was proved in \cite{HS02}. It was deduced from a corresponding stability result for the bounded Lipschitz metric (which is implicit in the proof, though it was not stated explicitly), namely
\begin{equation}\label{4.16} 
d_H(K,L') \le \gamma\,d_{bL}(S_{n-1}(K,\cdot),S_{n-1}(L,\cdot))^{1/n}
\end{equation}
for a suitable translate $L'$ of $L$, with a constant $\gamma$ as above. It appears that the H\"older continuity (\ref{4.14}) is, in principle, a more elementary fact than its reverse, the stability estimate (\ref{4.16}), and should therefore have preceded it.

\section{Notation and preliminaries}\label{sec2}

We have to use several definitions and results from geometric measure theory, therefore we choose most of our notation as in Federer's \cite{Fed69} book, in order to facilitate the comparison. For example, we denote the scalar product in $\R^n$ by $\bullet$ and the induced norm by $|\cdot|$. The same notation is used also for other Euclidean spaces which will come up in the following. We identify $\R^n$ and its dual space via the given scalar product. 

Next we recall some notation and basic facts from multilinear algebra. Let $V$ be finite-dimensional real vector spaces. Then 
$\bigwedge_m V$, for $m\in\N_0$, denotes the vector space of {\em $m$-vectors} of $V$, and $\bigwedge^m V$ is  the vector space of all  $m$-linear alternating maps from $V^m$ to $\R$, whose elements are called {\em $m$-covectors}. The map $\bigwedge^m V \to\text{Hom}(\bigwedge_mV,\R)$, which assigns to $f\in \bigwedge^m V$ the homomorphism $v_1\wedge \ldots\wedge v_m\mapsto f(v_1,\ldots,v_m)$, allows us to identify $\bigwedge^m V$ and $\text{Hom}(\bigwedge_mV,\R)$. By this identification,  the dual pairing of elements $a\in\bigwedge_m V$ and $\varphi\in \bigwedge^m (V,\R)$ can be defined by $\langle a,\varphi\rangle :=\varphi(a)$.  If $V'$ is another finite-dimensional vector space and  $f:V\to V'$ is a linear map, then a linear map $\bigwedge_mf:\bigwedge_mV\to\bigwedge_mV'$ is determined by 
$$
(\rmwedge_mf)(v_1\wedge\ldots\wedge v_m)=f(v_1)\wedge\ldots\wedge f(v_m),
$$
for all $v_1,\ldots,v_m\in V$.

Given an inner product space $(V,\bullet)$ with norm $|\cdot|$ we obtain an inner product on $\bigwedge_mV$. For 
$\xi,\eta\in \bigwedge_mV$ with $\xi=v_1\wedge\ldots\wedge v_m$ and $\eta=w_1\wedge\ldots\wedge w_m$, where 
$v_i,w_j\in V$, we define 
$$
\xi\bullet \eta=\det\left(\langle v_i,w_j\rangle_{i,j=1}^m\right).
$$
This is independent of the particular representation of $\xi,\eta$. For general $\xi,\eta\in \bigwedge_mV$ the inner product is defined by linear extension, and then we put $|\xi|:=\sqrt{\xi\bullet \xi}$ for $\xi\in \bigwedge_mV$. If $(b_1,\ldots,b_n)$ is an orthonormal basis of $V$, then the $m$-vectors $b_{i_1}\wedge\ldots\wedge b_{i_m}$ with $1\le i_1<\ldots<i_m\le n$ form an orthonormal basis of $\bigwedge_mV$. Moreover, if $\xi\in \bigwedge_pV$ or $\eta\in \bigwedge_qV$ is simple, then 
\begin{equation}\label{eq2}
|\xi\wedge \eta|\le |\xi|\, |\eta|.
\end{equation}
Let $(b_1,\ldots,b_n)$ be an orthonormal basis of $V$, and let $(b_1^*,\ldots,b_n^*)$ be the dual basis in $V^*= \bigwedge^1 V$. We endow $\bigwedge^m V$ with the inner product for which the vectors $b_{i_1}^*\wedge\ldots\wedge b_{i_m}^*$, for $1\le i_1<\ldots<i_m\le n$, are an orthonormal basis.  Then 
\begin{equation}\label{eq3}
|\langle \xi,\Phi\rangle|\le |\xi|\,|\Phi|
\end{equation} 
for $\xi\in\bigwedge_mV$ and $\Phi\in \bigwedge^mV$.  

The preceding facts are essentially taken from \cite[Section 1.7]{Fed69}. 

Let $V$ be an $n$-dimensional inner product space. Then {\em comass} and {\em mass} are defined as in \cite[Section 1.8]{Fed69}. In particular, for $\Phi\in\bigwedge^m V$ the comass $\|\Phi\|$ of $\Phi$ satisfies $\|\Phi\|=|\Phi|$ if $\Phi$ is simple. Moreover, for $\xi\in\bigwedge_m V$ the mass $\|\xi\|$ of $\xi$ satisfies  $\|\xi\|=|\xi|$ if $\xi$ is simple.

\vspace{2mm}

Now we turn to convex bodies. For notions from the theory of convex bodies which are not explained here, we refer to \cite{Sch93}. Let $K\in\Kn$. The metric projection, which is denoted by $p(K,\cdot)$, maps $\R^n$ to $ K$, and $u(K,x):=|x-p(K,x)|^{-1}(x-p(K,x))$ is defined for $x\in\R^n\setminus K$. Let $\partial K$ denote the topological boundary of $K$. The map $F:\partial K_1\to\R^n\times \Sn$ given by $F(x):= (p(K,x),u(K,x))$ is bi-Lipschitz, and the image is the {\em normal bundle} $\Nor K$ of $K$, which is an $(n-1)$ rectifiable subset of $\R^{2n}$. Hence, 
for $\mathcal{H}^{n-1}$-almost all $(x,u)\in\Nor K$, the set of $(\mathcal{H}^{n-1}\fed \Nor K,n-1)$ approximate tangent vectors at $(x,u)$ is an $(n-1)$-dimensional linear subspace of $\RR^{2n}$, which is denoted by $\textrm{Tan}^{n-1}(\mathcal{H}^{n-1}\fed \Nor K,(x,u))$. Let $\Pi_1:\R^n\times\R^n\to\R^n$, $(x,u)\mapsto x$, and $\Pi_2:\R^n\times\R^n\to\R^n$, $(x,u)\mapsto u$, be  projection maps and $\Omega_n$ the volume form for which $\Omega_n(e_1,\ldots,e_n)=1$ for the standard basis $(e_1,\ldots,e_n)$ of $\R^n$. 
The following statements hold for $\mathcal{H}^{n-1}$-almost all $(x,u)\in\Nor K$. First, we can choose an orthonormal basis $(a_1(x,u),\ldots,a_{n-1}(x,u))$ of $\textrm{Tan}^{n-1}(\mathcal{H}^{n-1}\fed \Nor K,(x,u))$ such that the $(n-1)$-vector $a_K(x,u):=a_1(x,u)\wedge\ldots\wedge a_{n-1}(x,u)$ satisfies
\begin{equation}\label{orientations}
\left\langle \rmwedge_{n-1}(\Pi_1+\varrho\,\Pi_2) a_K(x,u)\wedge u,\Omega_n\right\rangle>0
\end{equation}
for all $\varrho>0$ and thus determines an orientation of $\textrm{Tan}^{n-1}(\mathcal{H}^{n-1}\fed \Nor(K),(x,u))$. 
Then the {\em normal cycle} associated with the convex body $K$ is the $(n-1)$-dimensional current in 
$\R^{2n}$ which is defined by 
$$
T_K:=\left(\mathcal{H}^{n-1}\fed\Nor K\right)\wedge a_K.
$$
More generally, we can define
$$
T_K(\varphi)=\int_{\Nor K}\langle a_K(x,u),\varphi(x,u)\rangle\, \mathcal{H}^{n-1}(d(x,u))
$$
for all $\mathcal{H}^{n-1}\fed \Nor K$-integrable functions $\varphi:\R^n\times\R^n\to \bigwedge^{n-1}\R^{2n}$. Here we use that $T_K$ is a rectifiable current, which has compact support, and thus $T_K$ can be defined for a larger class of functions than just for smooth differential forms. 

Sometimes it is convenient to work with the orthonormal basis of the approximate tangent space of $\Nor K$ at $(x,u)$ that is given by  
$$
a_i(x,u):=\left( \alpha_i(x,u)\,b_i(x,u),\sqrt{1-\alpha_i(x,u)^2}\,b_i(x,u)\right),
$$
where $(b_1(x,u),\ldots,b_{n-1}(x,u))$ is a suitably chosen orthonormal basis of $u^\perp$ (the orthogonal complement of the linear subspace spanned by $u$) such that  $(b_1,\ldots,b_{n-1},u)$ has the same orientation as the standard basis $(e_1,\ldots,e_n)$ of $\R^n$, and $\alpha_i(x,u)\in [0,1]$ for $i=1,\ldots,n-1$. 
Note that the dependence of $a_i,b_i,\alpha_i$ on $K$ is not made explicit by our notation. The data $b_i, \alpha_i$, $i=1,\ldots,n-1$, are essentially uniquely determined (cf.~\cite[Proposition 3 and Lemma 2]{RZ05}). Moreover, we can assume that $b_i(x+\varepsilon u,u)=b_i(x,u)$, independent of $\varepsilon>0$, where $(x,u)\in\Nor K$ and $(x+\varepsilon u,u)\in\Nor K_\varepsilon$  with $K_\varepsilon:=K+\varepsilon B^n$. However, in general $\alpha_i(x+\varepsilon u,u)$ is not independent of $\varepsilon$ and will be positive for $\varepsilon>0$. See \cite{Zae86, RZ01, RZ05, Hug95, Hug98}  for a geometric 
description of the numbers $\alpha_i(x,u)$ in terms of generalized curvatures and for arguments establishing the facts stated here.

\section{Proof of Theorem \ref{T1}}\label{sec3}

In order to obtain an upper bound for $|T_K-T_L|$, we first establish an upper bound for $|T_{A_\varepsilon}-T_A|$, for $A\in\{K,L\}$ and $\varepsilon\in [0,1]$, which is done in Lemma \ref{Lemma4.3}. Then we derive an upper bound for  $|T_{K_\varepsilon}-T_{L_\varepsilon}|$ under the assumption that the Hausdorff distance of $K$ and $L$ is sufficiently small. This bound is provided in Lemma \ref{Lemma4.6}, which in turn is based on four preparatory lemmas.

\begin{lemma}\label{Lemma4.3} 
Let $K\in\Kn$ and $\varepsilon\in [0,1]$. Let $\varphi\in \mathcal{E}^{n-1}(\R^{2n})$. Then 
$$
\left|T_{K_{\varepsilon}}(\varphi)-T_K(\varphi)\right|\le C(K,\varphi)\, \varepsilon,
$$
where $C(K,\varphi)$ is a real constant, which depends on the maximum and the Lipschitz constant of $\varphi$ on $K_1\times \Sn$ and on  $\mathcal{H}^{n-1}(\partial K_1)$.
\end{lemma}

\begin{proof}
We consider the bi-Lipschitz map
$$
F_\varepsilon:\Nor K\to\Nor K_\varepsilon,\qquad (x,u)\mapsto (x+\varepsilon u,u).
$$
The extension of $F_\varepsilon$ to all $(x,u)\in\R^{2n}$ by $F_\varepsilon(x,u):= (x+\varepsilon u,u)$ is differentiable for all $(x,u)\in\R^{2n}$. By \cite[Theorem 3.2.22 (1)]{Fed69}, for $\mathcal{H}^{n-1}$-almost all $(x,u)\in\Nor K$ the approximate Jacobian of $F_\varepsilon$ satisfies 
\begin{equation}\label{eq2.1}
\textrm{ap}\,J_{n-1}F_\varepsilon (x,u)=\left\|\rmwedge_{n-1}\textrm{ap}\,DF_\varepsilon (x,u) a_K(x,u)\right\|>0,
\end{equation}
and the simple orienting $(n-1)$-vectors $a_K(x,u)$ and $a_{K_\varepsilon}(x+\varepsilon u,u)$ are related by  
\begin{equation}\label{eq2.2}
a_{K_\varepsilon}(x+\varepsilon u,u)=\frac{\bigwedge_{n-1}\textrm{ap}\,DF_\varepsilon (x,u) a_K(x,u)}{\left\|\bigwedge_{n-1}\textrm{ap}\,DF_\varepsilon (x,u) a_K(x,u)\right\|}.
\end{equation}
It follows from \eqref{orientations} that the orientations coincide.  
Thus, first using the coarea theorem \cite[Theorem 3.2.22]{Fed69} and then \eqref{eq2.1} and \eqref{eq2.2}, we get
\begin{align*}
T_{K_\varepsilon}(\varphi)&=\int_{\Nor K_\varepsilon}\langle a_{K_\varepsilon},\varphi\rangle\, \D\mathcal{H}^{n-1}\\
&=\int_{\Nor K}\langle a_{K_\varepsilon}\circ F_\varepsilon (x,u),\varphi\circ F_\varepsilon (x,u)\rangle\,
\textrm{ap}\,J_{n-1}F_\varepsilon (x,u)\, \mathcal{H}^{n-1}(\D (x,u))\\
&=\int_{\Nor K}\left\langle \rmwedge_{n-1}\textrm{ap}\,DF_\varepsilon (x,u) a_K(x,u),\varphi\circ F_\varepsilon (x,u)\right\rangle \mathcal{H}^{n-1}(\D (x,u)).
\end{align*}
By the triangle inequality, we obtain
\begin{eqnarray*}
\left|T_{K_\varepsilon}(\varphi)-T_K(\varphi)\right| &\le& \int_{\Nor K} \Big\{\left|\left\langle\left( \rmwedge_{n-1} \textrm{ap}\,DF_\varepsilon (x,u)-\rmwedge_{n-1}\textrm{id}\right) a_K(x,u),\varphi\circ F_\varepsilon (x,u)\right\rangle\right|\\
& & +\left|\left\langle a_K(x,u),\varphi(x+\varepsilon u,u)-\varphi(x,u)\right\rangle\right|\Big\}\,
 \mathcal{H}^{n-1}(\D (x,u)).
\end{eqnarray*}
We have
\begin{eqnarray*}
& & \left|\left\langle\left( \rmwedge_{n-1}\textrm{ap}\,DF_\varepsilon (x,u)-\rmwedge_{n-1}\textrm{id}\right) a_K(x,u),\varphi\circ F_\varepsilon (x,u)\right\rangle\right|\\
& & \le |\varphi(x+\varepsilon u,u)|\, 
\left|\left( \rmwedge_{n-1}\textrm{ap}\,DF_\varepsilon (x,u)-\rmwedge_{n-1}\textrm{id}\right) a_K(x,u)\right|,
\end{eqnarray*}      
where we  used \eqref{eq3}. Now $a_K(x,u)$ is of the form $\rmwedge_{i=1}^{n-1}(v_i,w_i)$ with suitable $(v_i,w_i)\in \R^{2n}$ and $|v_i|^2+|w_i|^2=1$. Moreover, we have $DF_\varepsilon(x,u)(v,w)=(v+\varepsilon w,w)$, for all $(v,w)\in \R^{2n}$. Writing $z_i^0:=v_i$, $z_i^1:=w_i$, we have
\begin{align*}
& \left| \left( \rmwedge_{n-1}\textrm{ap}\,DF_\varepsilon (x,u)-\rmwedge_{n-1}\textrm{id}\right) a_K(x,u)\right|
 =\left|\bigwedge_{i=1}^{n-1}(v_i+\varepsilon w_i,w_i)-\bigwedge_{i=1}^{n-1}(v_i,w_i)\right| \\
&= \left| \sum_{\alpha_i\in\{0,1\}} \varepsilon^{\sum\alpha_i} \bigwedge_{i=1}^{n-1}(z_i^{\alpha_i},w_i) - \bigwedge_{i=1}^{n-1}(z_i^0,w_i)\right|
 \le \varepsilon \sum_{\alpha_i\in\{0,1\}, \,\sum \alpha_i\ge 1} \, \left|\bigwedge_{i=1}^{n-1}(z_i^{\alpha_i},w_i) \right| \le c(n)\varepsilon,
\end{align*}
where we used \eqref{eq2} and the fact that $|(v_i,w_i)|=1$ and $|(w_i,w_i)|\le 2$. We deduce that 
$$
 |\varphi(x+\varepsilon u,u)|\, \left|\left( \rmwedge_{n-1}\textrm{ap}\,DF_\varepsilon (x,u)-\rmwedge_{n-1}\textrm{id}\right) a_K(x,u)\right|
\le C_1(K,\varphi)\varepsilon.
$$
Furthermore, again by \eqref{eq3} we get
$$
\left|\left\langle a_K(x,u),\varphi(x+\varepsilon u,u)-\varphi(x,u)\right\rangle\right|
\le |\varphi(x+\varepsilon u,u)-\varphi(x,u)|\le C_2(K,\varphi)\, \varepsilon .
$$
Thus we conclude that 
$$
\left|T_{K_\varepsilon}(\varphi)-T_K(\varphi)\right|\le C_3(K,\varphi)\,\varepsilon\, \mathcal{H}^{n-1}(\Nor K).
$$
Since $F:\partial K_1\to\Nor K$, $z\mapsto (p(K,z),z-p(K,z))$, is Lipschitz with Lipschitz constant bounded from above by 3, the assertion follows.
\end{proof}

A convex body $K\in\Kn$ is said to be $\varepsilon$-smooth (for some $\varepsilon>0$), if $K=K'+\varepsilon B^n$ for some $K'\in \Kn$. For a nonempty set $A\subset\R^n$, we define the distance from $A$ to $x\in\R^n$ by $d(A,x):= \inf\{|a-x|:a\in A\}$. The signed distance is defined by $d^*(A,x):=d(A,x)-d(\R^n\setminus A,x)$, $x\in\R^n$, if $A,\R^n \setminus A\neq \emptyset$. If $K$ is $\varepsilon$-smooth, then $\partial K$ has positive reach. More precisely, if $x\in \R^n$ satisfies $d(\partial K,x)<\varepsilon$, then there is a unique point $p(\partial K,x)\in\partial K$ such that $d(\partial K,x)=|p(\partial K,x)-x|$. 

\begin{lemma}\label{Lemma4.4}  Let $\varepsilon\in(0,1)$ and $\delta\in(0,\varepsilon/2)$. Let $K,L\in\Kn$ be $\varepsilon$-smooth and assume that $d_H(K,L)\le\delta$. Then
$$
p:\partial K\to \partial L,\qquad x\mapsto p(\partial L,x),
$$
is well-defined, bijective, bi-Lipschitz with $\textrm{\rm Lip}(p)\le \varepsilon/(\varepsilon-\delta)$, and $|p(x)-x|\le \delta$ for all $x\in\partial K$. 
\end{lemma}

\begin{proof} Since $d_H(K,L)\le \delta$, we have $K\subset L+\delta B^n$, $L\subset K+\delta B^n$, and a separation 
argument yields that
\begin{equation}\label{eqneighbour}
\{x\in L:d(\partial L,x)\ge \delta\}\subset K.
\end{equation}
This shows that $\partial K\subset \{z\in\R^n:d(\partial L,z)\le \delta\}$ 
and therefore the map $p$ is well-defined on $\partial K$ and $|p(x)-x|\le \delta$ for all $x\in \partial K$.  By 
\cite[Theorem 4.8 (8)]{Fed59} it follows that $\textrm{Lip}(p)\le \varepsilon/(\varepsilon-\delta)$. Since $L$ is $\varepsilon$-smooth, 
for $y\in \partial L$ there is a unique exterior unit normal of $L$ at $y$, which we denote by $u=u(L,y)=:u_L(y)$. Put $y_0:=y-\varepsilon u$ and note that  $y_0+(\varepsilon-\delta)B^n\subset K\cap L$ by \eqref{eqneighbour}. Then  $x\in\partial K$ is uniquely determined by the condition 
$\{x\}=\left(y_0+[0,\infty)u\right)\cap \partial K$ and satisfies $p(x)=y$. This shows that $p$ is surjective. 

Now let $x_1,x_2\in\partial K$ satisfy $p(x_1)=p(x_2)=:p_0\in\partial L$. Since there is a ball $B$ of radius $\varepsilon$ with $p_0\in B\subset L$, the points $x_1,x_2\in\partial K$ are on the line through $p_0$ and the center of $B$. By \eqref{eqneighbour}, they cannot be on different sides of $p_0$, hence $x_1=x_2$. This shows that the map $p$ is also injective. If $d^*(\partial K,\cdot):\R^n\to\partial K$ denotes the signed distance function of $\partial K$, then $q:\partial L\to\partial K$, $z\mapsto z-d^*(\partial K,z)u_L(z)$, is the inverse of $p$. Since the signed distance function is Lipschitz, Lemma 
\ref{Lespaet} shows that $q$ is Lipschitz as well.
\end{proof}

The following lemma provides a simple argument for the fact that the spherical image map of an $\varepsilon$-smooth 
convex body is Lipschitz with Lipschitz constant at most $1/\varepsilon$. A less explicit assertion is contained 
in \cite[Hilfssatz 1]{Lei86}.

\begin{lemma}\label{Lespaet}
Let $K\in \Kn$ be $\varepsilon$-smooth, $\varepsilon>0$. Then the spherical image map 
$u_K$  is Lipschitz with Lipschitz constant 
$1/\varepsilon$.
\end{lemma}

\begin{proof} Let $x,y\in \partial K$, and define $u:=u_K(x)$, $v:=u_K(y)$. Then  
$$
x-\varepsilon u+\varepsilon v\in x-\varepsilon u+ \varepsilon B^n \subset K,
$$
and hence $( x-\varepsilon u+\varepsilon v-y)\bullet v \leq 0$. This 
yields
\begin{equation}\label{Gluno}
\varepsilon ( v-u)\bullet v\leq  (y-x)\bullet v.
\end{equation}
By symmetry, we also have $\varepsilon ( u-v)\bullet u\leq  (x-y)\bullet u$, and therefore
\begin{equation}\label{Gldue}
\varepsilon ( v-u)\bullet(-u)\leq ( y-x)\bullet (-u).
\end{equation}
Addition of  (\ref{Gluno}) and (\ref{Gldue}) yields
$$
\varepsilon\,|v-u|^{2}\leq ( y-x)\bullet(v-u)\leq |y-x|\,|v-u|,
$$
which implies the assertion. 
\end{proof}

\begin{lemma}\label{Lemma4.5}
Let $\varepsilon\in(0,1)$ and $\delta\in(0,\varepsilon/2)$. Let $K,L\in\Kn$ be $\varepsilon$-smooth and assume that $d_H(K,L)\le\delta$. Put $p(x):=p(\partial L,x)$ for $x\in\partial K$. Then
$$
G:\Nor K\to \Nor L,\qquad (x,u)\mapsto (p(x),u_L(p(x))),
$$
is  bijective, bi-Lipschitz with $\textrm{\rm Lip}(G)\le 2/(\varepsilon-\delta)\le 4/\varepsilon$, and $|G(x,u)-(x,u)|\le \delta+2\sqrt{\delta/\varepsilon}$ for all $(x,u)\in\Nor K$. 
\end{lemma}

\begin{proof} It follows from Lemma \ref{Lemma4.4} that $G$ is bijective.  Then, for $(x,u),(y,v)\in\Nor K$ we get
\begin{align*}
|G(x,u)-G(y,v)|&\le |p(x)-p(y)|+|u_L(p(x))-u_L(p(y))|\\
&\le \frac{\varepsilon}{\varepsilon-\delta}|x-y|+\frac{1}{\varepsilon}\frac{\varepsilon}{\varepsilon-\delta}|x-y|\le \frac{\varepsilon+1}{\varepsilon-\delta}
|x-y|\\
&\le \frac{2}{\varepsilon-\delta}|(x,u)-(y,v)|,
\end{align*}
where we have used again Lemma \ref{Lemma4.4} and Lemma \ref{Lespaet}. 
Let $x\in\partial K$ and $z:=p(x)\in\partial L$. We want to bound $ u_L(z)\bullet u_K(x)$ 
from below. If $x\notin L$, then $\textrm{conv}(\{x\}\cup(z-\varepsilon u_L(z)+(\varepsilon-\delta)B^n))\subset K$, 
and therefore 
$$
 u_L(z)\bullet u_K(x) \ge \frac{\varepsilon-\delta}{\varepsilon+\delta}\ge 1-\frac{2\delta}{\varepsilon}.
$$
If $x\in L$, then in a similar way we obtain
$$
 u_L(z)\bullet u_K(x) \ge \frac{\varepsilon-\delta}{\varepsilon}\ge 1-\frac{\delta}{\varepsilon},
$$
hence 
\begin{equation}\label{angle}
 u_L(z)\bullet u_K(x) \ge  1-\frac{2\delta}{\varepsilon}
\end{equation}
holds for all $x\in\partial K$. 
Thus 
$$
|u_L(z)-u_K(x)|\le 2\sqrt{\delta/\varepsilon},
$$
which finally implies that, for all $(x,u)\in \Nor K$, 
$$
|G(x,u)-(x,u)|\le |p(x)-x|+|u_L(p(x))-u_K(x)|\le \delta+2\sqrt{\delta/\varepsilon}.
$$
Since $G^{-1}:\Nor L\to\Nor K$ is given by $G^{-1}(z,u)=(q(z),u_K(q(z)))$, 
it follows that also $G^{-1}$ is Lipschitz.
\end{proof}

Next we show that, under the assumptions of the subsequent lemma, $\rmwedge_{n-1}DG(x,u)$ is an orientation preserving map from the approximate tangent space of $\Nor K$ to the approximate tangent space of $\Nor L$. It seems that a corresponding fact is not provided in the proofs of related assertions in the literature.

\begin{lemma}\label{orient}
Let $\varepsilon\in(0,1)$ and $\delta\in(0,\varepsilon/(4n))$. Let $K,L\in\Kn$ be $\varepsilon$-smooth 
and assume that $d_H(K,L)\le\delta$. Then, for $\mathcal{H}^{n-1}$-almost all $(x,u)\in \Nor K$, the $(n-1)$-vector 
$\rmwedge_{n-1}DG(x,u) a_K(x,u)\in \textrm{\rm Tan}^{n-1}(\mathcal{H}^{n-1}\fed \Nor L,G(x,u))$ has the same orientation as $a_L(G(x,u))$.
\end{lemma}

\begin{proof} 
Let $x\in\partial K$, $u:=u_K(x)$, and $\bar{x}:=p(x)$, hence $d(\partial L,x)=|x-\bar x|$. 
The orientation of $\textrm{Tan}^{n-1}(\partial K,x)$ is determined by an arbitrary orthonormal basis $b_1(x),\ldots,b_{n-1}(x)$ of $u^\perp$ with 
$\Omega_n(b_1(x),\ldots,b_{n-1}(x),u)=1$. Similarly, any orthonormal basis $\bar b_1(\bar x),\ldots,\bar b_{n-1}(\bar x),\bar u$ with $\bar u:=u_L(p(x))$ determines the orientation of 
$\textrm{Tan}^{n-1}(\partial L,p(x))$. Since $G$ is bi-Lipschitz, we can assume that $(x,u)\in\Nor K$ is such that all differentials exist that are encountered in the proof. Moreover, we can also assume that $\rmwedge_{n-1}DG(x,u) a_K(x,u)$ spans $ \textrm{\rm Tan}^{n-1}(\mathcal{H}^{n-1}\fed \Nor L,G(x,u))$, where we write again $G$ for a Lipschitz extension of the given map $G$ to $\R^{2n}$. In the following, we put $b_i:=b_i(x)$ and $\bar b_i:=\bar b_i(\bar x)$ for $i=1,\ldots,n-1$.

The differentials of the maps $\Nor K\to\partial K$, $(x,u)\mapsto x$, and $\partial L\to\Nor L$, $z\mapsto (z,u_L(z))$, are  orientation preserving, which follows for instance from the discussion at the end of Section 2. Hence, it remains to 
be shown that the differential of $p:\partial K\to\partial L$, $x\mapsto p(x)$, is orientation preserving, that is, 
$$
\Delta:=\Omega_n(Dp(x)(b_1),\ldots,Dp(x)(b_{n-1}),\bar u)>0.
$$

First, we assume that $x\neq \bar x$, that is, $x\notin \partial L$.
Since $Dp(x)(\bar u)=o$, we get
$$
Dp(x)(b_i)=\sum_{j=1}^{n-1}  b_i\bullet \bar b_j \, Dp(x)(\bar b_j),
$$
and thus 
$$
\Delta=\det(B)\,\Omega_n(Dp(x)(\bar b_1),\ldots,Dp(x)(\bar b_{n-1}),\bar u),
$$
where $B=(B_{ij})$ with $B_{ij}:=  b_i\bullet \bar b_j $ for $i,j\in\{1,\ldots,n-1\}$. We choose $\bar b_1,\ldots,\bar b_{n-1}$ as  principal directions of curvature of $\partial L$ at $\bar x=p(x)$. Then $Dp(x)(\bar b_i)=\tau_i\, \bar b_i$ with 
$$
\tau_i:=1-d(\partial L,x)k_i\left(\partial L,\bar x,\frac{x-\bar x}{|x-\bar x|}\right)>0,
$$ 
for $i=1,\ldots,n-1$. Here we use that $L$ is $\varepsilon$-smooth,  hence $\partial L$ has positive reach, $d(\partial L,x)<\varepsilon$ and 
$$
\left|k_i\left(\partial L,\bar x,\frac{x-\bar x}{|x-\bar x|}\right)\right|\le 1/\varepsilon.
$$  
Hence it follows that $\Delta>0$ if we can show that $\det(B)>0$. Let $\tilde B=(\tilde B_{ij})$ be defined by $\tilde B_{ij}:=b_{ij}$, $\tilde B_{in}:=  b_i \bullet \bar u $, 
$\tilde B_{nj}:=  u\bullet \bar b_j $, and $\tilde B_{nn}:=  u\bullet \bar u $, for $i,j\in\{1,\ldots,n-1\}$.  
Then
\begin{align*}
1&=\Omega_n(b_1,\ldots,b_{n-1},u)\, \Omega_n(\bar b_1,\ldots,\bar b_{n-1},\bar u)=\det(\tilde B)\\
&\le  u\bullet \bar u\, \det(B)+\sum_{i=1}^{n-1}|  b_i\bullet \bar u |\cdot 1 
\le   u \bullet \bar u  \det(B)+\sqrt{n-1}\sqrt{1-  (u\bullet \bar u)  ^2}.
\end{align*}
From \eqref{angle} and our assumptions, we get $u\bullet \bar u\ge 1-(2\delta)/\varepsilon\ge 1-1/(2n)$, 
and therefore 
$$
\sqrt{1-  (u\bullet \bar u)  ^2}\le \sqrt{1/n}.
$$
Thus 
$$
1<  u\bullet \bar u  \,\det(B) +1,
$$
which implies that $\det(B)>0$. 

Finally, we have to consider the case where $x\in\partial L$. For $\mathcal{H}^{n-1}$-almost all $x\in \partial K\cap \partial L$, we have 
$\textrm{\rm Tan}^{n-1}(\mathcal{H}^{n-1}\fed (\partial K\cap\partial L),x)=u^\perp$ and  $Dp(x)=\textrm{id}_{u^\perp}$, since $p(z)=z$ for all $z\in \partial K\cap\partial L$. Hence, 
$\Delta=\Omega_n(b_1,\ldots,b_{n-1},\bar u)=u\bullet \bar u>0$.
\end{proof}

\begin{lemma}\label{Lemma4.6} 
Let $\varepsilon\in(0,1)$ and $\delta\in(0,\varepsilon/(4n))$. Let $K,L\in\Kn$ be $\varepsilon$-smooth 
and assume that $d_H(K,L)\le\delta$. Let $M\subset\R^{2n}$ be a compact convex set containing $K_{1-\varepsilon}\times \Sn$ and $L_{1-\varepsilon}\times \Sn$ in its interior. Then
$$
|T_K(\varphi)-T_L(\varphi)|\le C(M,\varphi)\,\left(\frac{4}{\varepsilon}\right)^{n-1}\,\left(\delta+2\sqrt{\delta/\varepsilon}\right)
$$
for $\varphi\in\mathcal{E}^{n-1}(\R^{2n})$, where $C(M,\varphi)$ is a constant which depends 
on the sup-norm and the Lipschitz constant of $\varphi$ on $M$, and on $\mathcal{H}^{n-1}(\partial K_1)$.
\end{lemma}

\begin{proof} Let $G$ be as in Lemma \ref{Lemma4.5} (or a Lipschitz extension to the whole space with the same Lipschitz constant). 
Then \cite[Theorem 4.1.30]{Fed69} implies that 
$$
T_L=G_\sharp T_K,
$$
since $\rmwedge_{n-1}DG$ preserves the orientation of the orienting $(n-1)$-vectors, by Lemma \ref{orient}. (In \cite{RZ01} a corresponding fact is stated without further comment.) Recall the definitions of the dual flat metric $\mathbf{F}_{M}$ from \cite[4.1.12]{Fed69} and of the 
mass $\mathbf{M}$ (of a current) from \cite[p.~358]{Fed69}. 
Using \cite[4.1.14]{Fed69},  $\partial T_K=0$, the fact that 
$T_K$ has compact support contained in the interior of $M$ and Lemma \ref{Lemma4.5}, we get
\begin{align*}
\mathbf{F}_{M}(T_L-T_K)&=\mathbf{F}_{M}(G_\sharp T_K-T_K)\le \mathbf{M}\,( T_K)\cdot \|G-\textrm{id}\|_{\Nor K,\infty}\cdot \left(\frac{4}{\varepsilon}\right)^{n-1}\\
&\le \mathcal{H}^{n-1}(\partial K_1)\, \left( \frac{4}{\varepsilon}\right)^{n-1}\, \left(\delta+2\sqrt{\delta/\varepsilon}\right),
\end{align*}
where $\|G-\textrm{id}\|_{\Nor K,\infty}:=\sup\{|G(x,u)-(x,u)|:(x,u)\in\Nor K\}$. The assertion now follows from the 
definition of  $\mathbf{F}_{M}$, since $\|d\varphi\|$ can be bounded in terms of the sup-norm  and the 
Lipschitz constant of $\varphi$ on $M$.
\end{proof}

Now we are in a position to complete the proof of Theorem \ref{T1}.

\begin{proof}[Proof of Theorem $\ref{T1}$] Let $\varphi\in\mathcal{E}^{n-1}(\R^{2n})$. Let $\delta:=d_H(K,L)>0$ and $\varepsilon:=\delta^{\frac{1}{2n+1}}$.  
Assume that $\delta<(4n)^{-\frac{2n+1}{2n}}$. Then  $\delta<\varepsilon/(4n)$.  Lemma \ref{Lemma4.3} implies that
\begin{align*}
\left|T_K(\varphi)-T_{K_{\varepsilon}}(\varphi)\right|&\le C(M,\varphi)\, \varepsilon,\\
\left|T_L(\varphi)-T_{L_{\varepsilon}}(\varphi)\right|&\le C(M,\varphi)\, \varepsilon.
\end{align*}
Since $K_\varepsilon$ and $L_\varepsilon$ are $\varepsilon$-smooth, $d_H(K_\varepsilon,L_\varepsilon)=\delta$, 
$(K_{\varepsilon})_{1-\varepsilon}=K$ and $(L_{\varepsilon})_{1-\varepsilon}=L$, Lemma \ref{Lemma4.6} shows that
$$
|T_{K_{\varepsilon}}(\varphi)-T_{L_{\varepsilon}}(\varphi)|\le  C(M,\varphi)\,\left(\frac{4}{\varepsilon}\right)^{n-1}\,\left(\delta+2\sqrt{\delta/\varepsilon}\right).
$$
The triangle inequality then yields
$$
|T_K(\varphi)-T_L(\varphi)|\le  C_4(M,\varphi)\,\left(\varepsilon+\frac{\delta}{\varepsilon^{n-1}}+\frac{1}{\varepsilon^{n-1}}\sqrt{\frac{\delta}{\varepsilon}}\right)
\le  C_5(M,\varphi)\, \delta^{\frac{1}{2n+1}}.
$$
If $\delta\ge (4n)^{-\frac{2n+1}{2n}}$, we simply adjust the constant.
\end{proof}

\section{Proof of Theorem \ref{T2}}\label{sec4}

In the theory of convex bodies, Federer's curvature measures are supplemented by the area measures, which are finite Borel measures on the unit sphere (and were, in fact, introduced more than 20 years earlier). The two series of measures are generalized by the support measures. We briefly recall their definition (see \cite[Chap.~4]{Sch93}). 

For $K\in\Kn$, we have already used the metric projection $p(K,\cdot)$ and the vector function $u(K,x):= (x-p(K,x))/d(K,x)$, where $d(K,x):= |x-p(K,x)|$ denotes the distance of the point $x$ from $K$. We also write $p(K,\cdot)=:p_K$, $u(K,\cdot)=:u_K$ and $d(K,\cdot)=:d_K$. For $\rho>0$, we set $K^\rho:=K_\rho\setminus K$, where $K_\rho=K+\rho B^n$ is the already defined parallel body of $K$ at distance $\rho$. The product space $\R^n\times \Sn$ is denoted by $\Sigma$. For given $K$, the map $f_\rho:K^\rho\to\Sigma$ is defined by $f_\rho(x):= (p_K(x),u_K(x))$ for $x\in K^\rho$, and $\mu_{K,\rho}=\mu_\rho(K,\cdot)$ is the image measure of $\Ha^n\fed K^\rho$ under $f_\rho$. This is a finite Borel measure on $\Sigma$, concentrated on ${\rm Nor}\,K$. The {\em support measures} $\Lambda_0(K,\cdot),\dots,\Lambda_{n-1}(K,\cdot)$ of $K$ can be defined by
\begin{equation}\label{4.9}  
\mu_{K,\rho} =\sum_{i=0}^{n-1} \rho^{n-i}\kappa_{n-i}\Lambda_i(K,\cdot).
\end{equation}
Thus, the normalization is different from \cite[(4.2.4)]{Sch93}; the connection is given by $n\kappa_{n-i} \Lambda_i(K,\cdot)=\binom{n}{i}\Theta_i(K,\cdot)$. The support measures have the property of weak continuity: if a sequence $(K_j)_{j\in\N}$ of convex bodies converges to a convex body $K$ in the Hausdorff metric, then the sequence $(\Lambda_i(K_j,\cdot))_{j\in\N}$ converges weakly to $\Lambda_i(K,\cdot)$. The topology of weak convergence can be metrized by the bounded Lipschitz metric $d_{bL}$ or the L\'{e}vy--Prokhorov metric $d_{LP}$ (see, e.g., Dudley \cite[Sec.~11.3]{Dud02}). Therefore, the question arises whether the weak continuity of the support measures can be improved to H\"older continuity with respect to one of these metrics. 

For bounded real functions $f$ on $\Sigma$ we define
$$ \|f\|_L:=\sup_{x\not= y}\frac{|f(x)-f(y)|}{|x-y|},\qquad \|f\|_\infty:=\sup_x |f(x)|.$$
For finite Borel measures $\mu,\nu$ on $\Sigma$, their {\em bounded Lipschitz distance} is defined by
$$ d_{bL}(\mu,\nu):= \sup\left\{ \left|\int_\Sigma f\,\D\mu- \int_\Sigma f\,\D\nu\right|:f:\Sigma\to\R,\; \|f\|_L\le 1,\;\|f\|_\infty\le 1\right\}.$$

The following lemma is modeled after Proposition 4.1 of Chazal, Cohen--Steiner and M\'{e}rigot \cite{CCM10}. Under the restriction to convex bodies, it extends the latter to the measures $\mu_{K,\rho}$.

\begin{lemma}\label{L4.1}
If $K,L\in\Kn$ are convex bodies and $\rho>0$, then
$$
d_{bL}(\mu_{K,\rho},\mu_{L,\rho}) \le \int_{K^\rho\cap L^\rho} |p_K-p_L|\,\D\Ha^n+ \int_{K^\rho\cap L^\rho} |u_K-u_L|\,\D\Ha^n +\Ha^n(K^\rho\triangle L^\rho),
$$
where $\triangle$ denotes the symmetric difference.
\end{lemma}

\begin{proof}
Let $f:\Sigma\to\R$ be a function with $\|f\|_L\le 1$ and $\|f\|_\infty\le 1$. Using the transformation formula for integrals and the properties of $f$, we obtain
\begin{eqnarray*}
& & \left|\int_\Sigma f\,\D\mu_{K,\rho} -\int_\Sigma f\,\D\mu_{L,\rho} \right|\\
& & = \left|\int_{K^\rho} f\circ(p_K,u_K)\,\D\Ha^n -\int_{L^\rho} f\circ(p_L,u_L)\,\D\Ha^n\right|\\
& & \le \int_{K^\rho\cap L^\rho} \left|f\circ(p_K,u_K)-f\circ(p_L,u_L)\right|\,\D\Ha^n\\
& & \hspace{4mm}+\int_{K^\rho\setminus L^\rho}  \left|f\circ(p_K,u_K)\right|\,\D\Ha^n +\int_{L^\rho\setminus K^\rho}  \left|f\circ(p_L,u_L)\right|\,\D\Ha^n \\
& & \le \int_{K^\rho\cap L^\rho} |(p_K,u_K)-(p_L,u_L)|\,\D\Ha^n + \int_{K^\rho\setminus L^\rho}  1\,\D\Ha^n +\int_{L^\rho\setminus K^\rho} 1\,\D\Ha^n \\
& & \le \int_{K^\rho\cap L^\rho} (|p_K-p_L|+|u_K-u_L|)\,\D\Ha^n +\Ha^n(K^\rho\triangle L^\rho),
\end{eqnarray*}
from which the assertion follows.
\end{proof}

\noindent{\em Proof of Theorem} \ref{T2}. We assume that $K,L\in\Kn$ and  $d_H(K,L)=:\delta<1$. Let $R$ be the radius of a ball containing $K_2$ and $L_2$. For $0<\rho\le 1$ we use Lemma \ref{L4.1} (where for convex bodies, the estimation of the first and the third term on the right-hand side is easier than for the case of general compact sets considered in \cite{CCM10}). First, from Lemma 1.8.9 in \cite{Sch93} we get
\begin{equation}\label{4.10} 
\int_{K^\rho\cap L^\rho} |p_K-p_L|\,\D\Ha^n\le\sqrt{5D}\Ha^n(K^\rho\cap L^\rho)\sqrt{\delta} \le C_1(R)\sqrt{\delta},
\end{equation}
where $D={\rm diam}(K_\rho\cup L_\rho)$ and the constant $C_1(R)$ depends only on $R$.

About the distance function $d_K$, it is well known that
$$ \sup_{x\in\R^n}|d_K(x)-d_L(x)|=d_H(K,L)=\delta$$
and that
$$ \nabla d_K=u_K \qquad\mbox{on }\R^n\setminus K.$$
Therefore, it follows immediately from Theorem 3.5 of Chazal, Cohen--Steiner and M\'{e}rigot \cite{CCM10} (applied to $E={\rm int}(K^\rho\cap L^\rho)$) that
\begin{equation}\label{4.11} 
\int_{K^\rho\cap L^\rho} |u_K-u_L|\,\D\Ha^n\le C_2(R)\sqrt{\delta}.
\end{equation}

For the estimation of $\Ha^n(K^\rho\triangle L^\rho)$, let $x\in K^\rho\setminus L^\rho$; then $x\in K_\rho\setminus K$ and $x\notin L_\rho\setminus L$. If $x\in L$, then $d(K,x)\le\delta$, hence $x\in K_\delta\setminus K$. If $x\notin L$, then $x\notin L_\rho$ but $x\in K_\rho$, $K_\rho\subset(L_\delta)_\rho=L_{\rho+\delta}$, and hence $x\in L_{\rho+\delta}\setminus L_\rho$. It follows that
$$ K^\rho \setminus L^\rho \subset (K_\delta\setminus K) \cup (L_{\rho+\delta}\setminus L_\rho)$$
and hence 
\begin{eqnarray*}
\Ha^n(K^\rho\setminus L^\rho) &\le & \Ha^n(K_\delta)-\Ha^n(K) + \Ha^n(L_{\rho+\delta}) -\Ha^n(L_\rho)\\
&\le &C_3(R)\delta\le C_3(R)\sqrt{\delta}.
\end{eqnarray*}
Here $K$ and $L$ can be interchanged, and together with (\ref{4.10}), (\ref{4.11}) and Lemma \ref{L4.1} this gives
\begin{equation}\label{4.12}
d_{bL}(\mu_{K,\rho},\mu_{L,\rho}) \le  C_4(R)\sqrt{\delta}.
\end{equation}

To deduce an estimate for the support measures, we apply the usual procedure (e.g., \cite{Sch93}, p. 202) and choose in (\ref{4.9}) for $\rho$ each of the $n$ fixed values $\rho_j=j/n$, $j=1,\dots,n$, and solve the resulting system of linear equations (which has a non-zero Vandermonde determinant), to obtain representations
$$ \Lambda_i(K,\cdot) = \sum_{j=1}^n a_{i,j}\mu_{K,\rho_j},\qquad i=0,\dots,n-1,$$
with constants $a_{ij}$ depending only on $i,j$. Using the definition of the bounded Lipschitz metric, we deduce that
\begin{equation}\label{4.13} 
d_{bL}(\Lambda_i(K,\cdot),\Lambda_i(L,\cdot)) \le \sum_{j=1}^n|a_{ij}|d_{bL}(\mu_{K,\rho_j},\mu_{L,\rho_j}) \le C(R)\sqrt{\delta}.
\end{equation}
This completes the proof of Theorem \ref{T2}. \qed

\noindent Authors' addresses:\\[2mm]
Daniel Hug\\
Karlsruhe Institute of Technology, Department of Mathematics\\
D-76128 Karlsruhe, Germany\\
E-mail: daniel.hug@kit.edu\\[3mm]
Rolf Schneider\\
Mathematisches Institut, Albert-Ludwigs-Universit{\"a}t\\
D-79104 Freiburg i. Br., Germany\\
E-mail: rolf.schneider@math.uni-freiburg.de

\end{document}